\documentclass[11pt, onesided, reqno]{amsart}

\usepackage{amssymb}
\usepackage{amsmath}
\usepackage{color}
\usepackage{enumerate}
\usepackage{graphicx}

\evensidemargin\oddsidemargin

\begin{document}

\setcounter{page}{1}

\newtheorem{PROP}{Proposition}
\newtheorem{REMS}{Remark}
\newtheorem{LEM}{Lemma}
\newtheorem{THEA}{Theorem A\!\!}
\renewcommand{\theTHEA}{}
\newtheorem{THEB}{Theorem B\!\!\!}
\renewcommand{\theTHEB}{}

\newtheorem{theorem}{Theorem}
\newtheorem{proposition}[theorem]{Proposition}
\newtheorem{corollary}[theorem]{Corollary}
\newtheorem{lemma}[theorem]{Lemma}
\newtheorem{assumption}[theorem]{Assumption}

\newtheorem{definition}[theorem]{Definition}
\newtheorem{hypothesis}[theorem]{Hypothesis}

\theoremstyle{definition}
\newtheorem{example}[theorem]{Example}
\newtheorem{remark}[theorem]{Remark}
\newtheorem{question}[theorem]{Question}

\newcommand{\eqnsection}{
\renewcommand{\theequation}{\thesection.\arabic{equation}}
    \makeatletter
    \csname  @addtoreset\endcsname{equation}{section}
    \makeatother}
\eqnsection

\def\E{\mathbb{E}}
\def\Z{\mathbb{Z}}
\def\N{\mathbb{N}}
\def\Q{\mathbb{Q}}
\def\R{\mathbb{R}}
\def\P{\mathcal{P}}
\def\Pb{\mathbb{P}}
\def\C{\mathbb{C}}
\def\A{\mathcal{A}}
\def\cC{\mathcal{C}}
\def\F{\mathcal{F}}
\def\S{\mathcal{S}}
\def\W{\mathcal{W}}
\def\L{\mathcal{L}}
\def\G{\mathcal{G}}
\def\i{\text{i}}
\def\sgn{\text{sgn}}

\newcommand{\Un}{{\bf 1}}
\newcommand{\equi}{\mathop{\sim}\limits}
\def\={{\;\mathop{=}\limits^{\text{(law)}}\;}}
\def\d{{\;\mathop{=}\limits^{\text{(1.d)}}\;}}
\def\st{{\;\mathop{\geq}\limits^{\text{(st)}}\;}}

\def\cas{\stackrel{a.s.}{\longrightarrow}}
\def\claw{\stackrel{d}{\longrightarrow}}
\def\elaw{\stackrel{d}{=}}
\def\qed{\hfill$\square$}

\newcommand*\pFqskip{8mu}
\catcode`,\active
\newcommand*\pFq{\begingroup
        \catcode`\,\active
        \def ,{\mskip\pFqskip\relax}%
        \dopFq
}
\catcode`\,12
\def\dopFq#1#2#3#4#5{%
        {}_{#1}F_{#2}\biggl[\genfrac..{0pt}{}{#3}{#4};#5\biggr]%
        \endgroup
}

\newcommand*\pMqskip{8mu}
\catcode`,\active
\newcommand*\pMq{\begingroup
        \catcode`\,\active
        \def ,{\mskip\pMqskip\relax}%
        \dopMq
}
\catcode`\,12
\def\dopMq#1#2#3{%
        M\biggl[\genfrac..{0pt}{}{#1}{#2};#3\biggr]%
        \endgroup
}

\title[]{Persistence and exit times for some additive functionals of skew Bessel processes}
\author[Christophe Profeta]{Christophe Profeta}

\address{ Laboratoire de Math\'ematiques et Mod\'elisation d'Evry (LaMME),   Universit\'e d'Evry-Val-d'Essonne, UMR CNRS 8071, F-91037 Evry Cedex. {\em Email} : {\tt christophe.profeta@univ-evry.fr}}





\begin{abstract} Let $X$ be some homogeneous additive functional of a skew Bessel process $Y$. In this note, we compute the asymptotics of the first passage time of $X$ to some fixed level $b$, and study the position of $Y$ when $X$ exits a bounded interval $[a,b]$. As a by-product, we obtain the probability  that $X$ reaches the level $b$ before the level $a$. 
Our results extend some previous works on additive functionals of Brownian motion by Isozaki and Kotani for the persistence problem, and by Lachal for the exit time problem.
\end{abstract}
\maketitle
\section{Introduction}

\noindent
\subsection{Statement of the results}
Let $Y$ be a skew Bessel process with dimension $\delta \in [1,2)$ and skewness parameter $\eta \in (-1,1)$. $Y$ is a linear diffusion on $\R$ with scale function $s$ and speed measure $m$ given by :
$$
s(y)=\begin{cases}
 \frac{1-\eta}{2-\delta}\, y^{2-\delta} &\quad \text{for }y>0,\\
 \vspace{-.3cm}\\
- \frac{1+\eta}{2-\delta}\, |y|^{2-\delta} &\quad \text{for }y<0,\\
\end{cases}
\quad \text{ and }\quad  
m(dy)=\begin{cases}
\frac{1}{1-\eta} \,y^{\delta-1}dy &\quad \text{for }y>0,\\
 \vspace{-.3cm}\\
 \frac{1}{1+\eta} \,|y|^{\delta-1}dy &\quad \text{for }y<0.\\
\end{cases}
$$
Heuristically, $Y$ may be constructed by starting from a standard Bessel process and flipping independently each excursion to the negative half-line with probability $\frac{1-\eta}{2}$. The study of the stochastic differential equation satisfied by the skew Bessel process was undertaken by Blei \cite{Ble}, while its semi-group was computed for instance in Alili-Aylwin \cite{AlAy}. We also refer to Lejay \cite{Lej} for a nice account in the special case of skew Brownian motion, i.e when $\delta=1$.\\

\noindent
 For $\gamma>0$ and $c>0$, let us set $V_\gamma(y) =  |y|^{\gamma} \left(\Un_{\{y\geq0\}}  - c \Un_{\{y<0\}}\right)$ and consider the homogeneous additive functional 
$$X_t= \int_0^{t}V_\gamma(Y_u) du, \qquad\qquad t\geq0.$$
We denote by $\Pb_{(x,y)}$ the law of $(X,Y)$ when started from $(x,y)\in \R^2$, with the convention that $\Pb=\Pb_{(0,0)}$. The pair $(X,Y)$ is Markovian and satisfies the following scaling property : for any $k>0$, the law of $(X_{kt}, Y_{kt})_{t\geq0}$ under $\Pb_{(x,y)}$ is the same as that of $(k^{1+\frac{\gamma}{2}}X_{t}, k^{\frac{1}{2}}Y_{t})_{t\geq0}$ under $\Pb_{(x k^{-1-\frac{\gamma}{2}},y k^{-\frac{1}{2}})}$.
Define, for $b\in \R$, the stopping time
$$T_b = \inf\{t>0,\, X_t= b\}. $$
We start by computing the asymptotics of the survival function $\Pb_{(x,y)}(T_b> t)$ as $t\rightarrow+\infty$. Such studies are known as persistence problems, and have received a lot of interest recently. We refer to the survey \cite{AuSi} for a review of the mathematical literature (as well as several conjectures), and to \cite{BMS} for a more physics point of view and many applications.
\\

In the case of the integrated Brownian motion (i.e. $\{\delta=\gamma=c=1 \text{ and }\eta=0 \}$),  the rate of decay was computed by Groeneboom, Jongbloed  and Wellner \cite{GJW} using the explicit density of the pair $(T_b, Y_{T_b})$ obtained by Lachal \cite{LacT0}. We also refer to \cite{LacTn} for the study of the $n^{\text{th}}$ passage time of the integrated Brownian motion, and to \cite{Pro} for its asymptotics and some application to penalizations. The case of an additive functional of Brownian motion  (i.e. $\{\delta=1,\, \eta=0, \gamma>0  \text{ and } c>0 \}$) was then solved  by Isozaki and Kotani \cite{IsKo} using excursion theory and a Tauberian theorem.\\

We shall first extend these results to the case of skew Bessel processes. 
To this end, let us set 
$$\nu=\frac{2-\delta}{2+\gamma} \in \left(0,\, \frac{1}{2}\right)$$
and
$$\theta=  \frac{2+\gamma}{2\pi } \text{Arctan} \left( \frac{\sin\left(\nu\pi \right)}{  c^{\nu} \frac{1-\eta}{1+\eta} + \cos\left(\nu\pi \right) }\right)  \in \left(0, 1-\frac{\delta}{2}\right).$$
Define next the first hitting time of 0 by $Y$ :
$$ \sigma_0 = \inf\{t\geq0, \,Y_t=0\} $$
and consider the harmonic  function $h$ defined on $\R\times\R$ by 
\begin{equation}\label{eq:defh}
h(x,y) = 
\left\{
\begin{array}{cl}
\E_{(0,y)}\left[\left(x-X_{\sigma_0}\right)_+^{\frac{2\theta}{2+\gamma}}\right]& \text{ if $x\geq0$, }\\
0 &  \text{ if $x<0$ }
\end{array}\right.
\end{equation}
where $a_+ = \max(a,0)$.
The function $h$ is increasing in $x$, decreasing in $y$, and satisfies for $x>0$ the scaling property :
\begin{equation}\label{eq:hscale}
 h(x,y) = x^{\frac{2\theta}{2+\gamma}} h\left(1, y x^{-\frac{1}{2+\gamma}} \right).
 \end{equation}


\noindent
Our first result is the following asymptotics. 
\begin{theorem}\label{theo:A} 
Assume that  $\{x<b \text{ and }y\in \R\}$ or $\{x=b \text{ and }y<0\}$. Then there exists a constant $\kappa>0$ independent from $(x,y)$ such that  
$$ \Pb_{(x,y)}\left( T_b> t  \right)\sim  \kappa \, h(b-x,y)\, t^{-\theta},\qquad\qquad t\rightarrow +\infty.
$$
\end{theorem}

One may observe that the influences of the parameters $c$ and $\eta$ on the persistence exponent $\theta$ are similar. This is due to the fact that they play similar roles in the expression of $X$, i.e. they put different weights on the positive and negative excursions of $Y$. \\
 
\noindent The power decay of Theorem \ref{theo:A} is typical of self-similar processes, but computing the value of $\theta$ is generally a difficult task. Other examples for which $\theta$ is explicitly known are for instance the fractional Brownian motion (\cite[Section 3.3]{AuSi}), stable L\'evy processes (\cite[Section 2.2]{AuSi}) as well as their integrals \cite{PSPers}. Conversely, it still remains an open problem for instance to find the value of the persistence exponent for the integrated fractional Brownian motion, or for the twice integrated Brownian motion.\\

\noindent
Letting formally $c\rightarrow 0$ in Theorem \ref{theo:A}, we obtain $\theta= 1-\frac{\delta}{2}$, and by scaling, we may infer that
$$ \Pb\left( \int_0^{1} (Y_s)_+^{\gamma}\, ds <\varepsilon \right)  = \Pb\left( \int_0^{\varepsilon^{-\frac{2}{2+\gamma}}} \!(Y_s)_+^\gamma\, ds <1 \right) =  \Pb\left(T_1>\varepsilon^{-\frac{2}{2+\gamma}}\right)\equi_{\varepsilon\rightarrow 0}  \kappa\, \varepsilon^{\nu}.$$
In the standard Brownian case (when $\delta=\gamma=1$ and $\eta=0$), this exponent agrees with the conjecture in Janson \cite[eq. (261)]{JanArea}.\\

\noindent
The general idea of the proof of Theorem \ref{theo:A} is to work with the first zero of $Y$ after $X$ has crossed the level $b>0$ :
\begin{equation}\label{eq:defzetab}
  \zeta_b = \inf\{t>0,\, X_t\geq b \text{ and } Y_t= 0\}.
\end{equation}
This stopping time turns out to be easier to study than $T_{b}$, and yields far more simpler expressions. We shall then retrieve information on $T_b$ by applying the Markov property, see Section \ref{sec:2}. \\

We now turn our attention to the study of the exit time of $X$ from the interval $[a,b]$ with $a<0<b$~:
$$T_{ab}= \inf\{t>0, \, X_t\notin (a,b)\}.$$
In the Brownian case,   i.e. when $\{\delta=1, \,\eta=0,\, \gamma>0 \text{ and }c>0 \}$, the distribution of $Y_{T_{ab}}$ was computed by Lachal \cite{LacExit1, LacExit2} using excursion theory for the bivariate process $(X,Y)$.  Following his notation,  but changing the signs to keep positive parameters, we set 
$$A = \frac{(2+\gamma)^2}{2},$$
$$ \alpha = \frac{1}{\pi}\text{Arctan}\left(\frac{ \sin\left(\nu \pi\right)}{c^{-\nu}\frac{1+\eta}{1-\eta}+ \cos\left(\nu \pi\right)  }\right)\qquad \text{ and } \qquad\beta = \frac{2}{2+\gamma} \theta.$$
Notice that with these definitions
$$\alpha,\beta \in (0, \nu)\quad \text{ and }\quad \alpha+\beta=\nu.$$ In the following, we shall denote  by $I_\nu$ the modified Bessel function of the first kind and we recall the definition of the hypergeometric functions ${}_1F_1$  and ${}_2F_1$ (see for instance \cite[Chapters 9.1 and 9.2]{GrRy}): 
$$\pFq{1}{1}{a}{c}{z} = \sum_{n=0}^{+\infty}  \frac{(a)_n}{(c)_n} \frac{z^n}{n!}\qquad \text{ and }\qquad \pFq{2}{1}{a\quad b}{c}{z} = \sum_{n=0}^{+\infty}  \frac{(a)_n(b)_n}{(c)_n} \frac{z^n}{n!}
$$
where $(a)_n = a(a+1)\ldots (a+n-1)$ for $n\in \N$.


\begin{theorem}\label{theo:B} Assume that $a<x<b$. The probability density function of $Y_{T_{ab}}$ admits the expression~:
$$\Pb_{(x,0)}(Y_{T_{ab}}\in dz)/dz=  \left\{
\begin{array}{lr}
\displaystyle D_+
\left(\frac{x-a}{b-a}\right)^{\alpha} \frac{z^{\gamma+\alpha(\gamma+2)+\delta-1}}{(b-x)^{1-\beta}} \exp\left(-\frac{ z^{\gamma+2}}{A (b-x)}\right)&\\
\vspace{-.3cm}\\
\displaystyle\hspace{2cm} \times \pFq{1}{1}{\alpha}{ 1+\alpha}{\frac{ (x-a)z^{\gamma+2}}{A(b-x)(b-a)}} &  \text{if }z>0, \\
\vspace{-.1cm}\\
 \displaystyle D_-
\left(\frac{b-x}{b-a}\right)^{\beta} \frac{|z|^{\gamma+\beta(\gamma+2)+\delta-1}}{(x-a)^{1-\alpha}} \exp\left(-\frac{ c|z|^{\gamma+2}}{A (x-a)}\right)&\\
\vspace{-.3cm}\\
\displaystyle\hspace{2cm}\times \pFq{1}{1}{\beta}{ 1+\beta}{\frac{c (b-x)|z|^{\gamma+2}}{A(x-a)(b-a)}} &  \text{if }z<0,
\end{array}\right.
$$
where the constants $D_+$ and $D_-$ are given by
$$D_+ =  \frac{\Gamma(\nu)}{ A^{1-\beta}}\frac{\sin(\pi \beta)}{\pi \nu} \frac{2-\delta}{\Gamma(1+\alpha)}\quad \text{ and }\quad D_-=\Gamma(\nu) \left(\frac{c}{A}\right)^{1-\alpha}\frac{\sin(\pi \alpha)}{\pi \nu } \frac{2-\delta}{\Gamma(1+\beta)}.$$
\end{theorem}
Observe that the expressions we obtain are almost the same as those in \cite{LacExit1}, except for the occurrence of the parameter $\delta$. The general idea to prove Theorem \ref{theo:B} is similar to that of Theorem \ref{theo:A}. We shall first compute the law of $X_{\zeta_{ab}}$ where
\begin{equation}\label{eq:defzetaab}
\zeta_{ab}=\inf\{t>0,\, X_t\notin (a,b) \text{ and }Y_t=0\}
\end{equation}
and then retrieve the distribution of $Y_{T_{ab}}$ via a modified Laplace transform, see Section \ref{sec:3}. More generally, Theorem \ref{theo:B}  and the Markov property allow to express the law of $Y_{T_{ab}}$ also in the case $y\neq0$. 
\begin{theorem}\label{theo:C} Assume that $y\neq0$. The probability density function of $Y_{T_{ab}}$ is given as follows.
\begin{enumerate}
\item If $y>0$ and $a\leq x<b$ then :
\begin{multline*}
\Pb_{(x,y)}\left(Y_{T_{ab}}\in dz\right) = \frac{y^{2-\delta}}{A^\nu \Gamma(\nu)}  \int_0^{b-x} \Pb_{(x+u,0)}\left(Y_{T_{ab}}\in dz\right)u^{-\nu-1} e^{-\frac{y^{2+\gamma}}{Au}}du \\
+ \frac{2+\gamma}{A(b-x)} y^{1-\frac{\delta}{2}}z^{\gamma+\frac{\delta}{2}} e^{-\frac{z^{2+\gamma}+ y^{2+\gamma}}{A(b-x)}} I_\nu\left(2\frac{(zy)^{1+\frac{\gamma}{2}}}{A(b-x)}\right)\Un_{\{z>0\}} \,dz.  
\end{multline*}
\item If $y<0$ and $a< x\leq b$ then :
\begin{multline*}
\Pb_{(x,y)}\left(Y_{T_{ab}}\in dz\right) =\frac{|y|^{2-\delta}}{\Gamma(\nu)} \left(\frac{c}{A}\right)^{\nu} \int_0^{x-a} \Pb_{(x-u,0)}\left(Y_{T_{ab}}\in dz\right)u^{-\nu-1} e^{-\frac{c|y|^{2+\gamma}}{Au}}du\\
+ \frac{c(2+\gamma)}{A(x-a)} |y|^{1-\frac{\delta}{2}}|z|^{\gamma+\frac{\delta}{2}} e^{-c\frac{|z|^{2+\gamma}+ |y|^{2+\gamma}}{A(x-a)}} I_\nu\left(2\frac{c|zy|^{1+\frac{\gamma}{2}}}{A(x-a)}\right) \Un_{\{z<0\}}\, dz.
\end{multline*}
\end{enumerate}

\end{theorem}

\noindent
As a consequence of Theorems \ref{theo:B} and \ref{theo:C}, we may obtain the probability that $X$ reaches one level before the other one.

\begin{corollary}\label{cor:D}
The probability that $X$ hits the level $b$ before the level $a$ is given as follows.
\begin{enumerate}
\item If $a<x<b$ and $y=0$ :
$$\Pb_{(x,0)}\left(T_b<T_a\right)=\frac{\Gamma\left(\nu\right)}{\Gamma(1+\alpha)\Gamma\left(\beta\right)} \left(\frac{x-a}{b-a}\right)^{\alpha} 
 \pFq{2}{1}{\alpha, \;1-\beta}{1+\alpha}{\frac{x-a}{b-a}}.
$$
\item If $a<x\leq b$ and $y<0$ :
$$\Pb_{(x,y)}\left(T_b<T_a\right)=\frac{|y|^{2-\delta}}{\Gamma(\nu)} \left(\frac{c}{A}\right)^{\nu} \int_0^{x-a} \Pb_{(x-u,0)}\left(T_b<T_a\right)u^{-\nu-1} e^{-\frac{c|y|^{2+\gamma}}{Au}}du. $$
\item If $a\leq x<b$ and $y>0$ :
\begin{multline*}
 \Pb_{(x,y)}\left(T_b<T_a\right)=\frac{y^{2-\delta}}{A^\nu \Gamma(\nu)}   \left(\int_0^{b-x}\!\! \Pb_{(x+u,0)}\left(T_b<T_a\right)u^{-\nu-1}e^{-\frac{y^{2+\gamma}}{Au}}du+
\int_{b-x}^{+\infty} u^{-\nu-1}  e^{-\frac{y^{2+\gamma}}{Au}} du \right).
\end{multline*}
\end{enumerate}

\end{corollary}

\begin{remark}
We believe our results remain valid when $\delta\in (0,1)$, but our proofs unfortunately do not apply to this case as we need $Y$ to be a semimartingale in order to apply the It\^o-Tanaka formula, see Section \ref{sec:MM}.
\end{remark}

\noindent

\subsection{Explicit expressions for $h$} Before going to the proofs of the Theorems, we mention the following Lemma which allows to obtain an explicit expression for the harmonic function~$h$.

\begin{lemma}\label{lem:Xs0} The probability density function of $X_{\sigma_0}$ is given by :
$$\Pb_{(x,y)}\left(X_{\sigma_0}\in dz\right)/dz = \left\{
\begin{array}{lr}
\displaystyle  \frac{A^{-\nu} }{\Gamma\left(\nu\right)}  \frac{y^{2-\delta}}{(z-x)^{\nu+1}}  \exp\left(- \frac{y^{2+\gamma} }{A\, (z-x)} \right)\Un_{\{z>x\}} \, 
&  \text{if }y>0, \\
\vspace{-.1cm}\\
\displaystyle\frac{1}{\Gamma\left(\nu\right)}  \left(\frac{c}{A}\right)^{\nu} \frac{|y|^{2-\delta}}{|z-x|^{\nu+1}}  \exp\left(- \frac{c|y|^{2+\gamma} }{A|z-x|} \right) \Un_{\{z<x\}}\,   &  \text{if }y<0.\\
\end{array}\right.
$$
\end{lemma}

\begin{proof}
This result is classic : one may for instance inverse the Laplace transform of $X_{\sigma_0}$ which was computed by Cetin \cite[Corollary 2.1]{Cet}. Another option is to use 
Dufresne's formula combined with the Lamperti transform for the geometric Brownian motion, see \cite[p.15-16]{Yor}.

\end{proof}

Lemma \ref{lem:Xs0} allows to give explicit expressions for $h$ in terms of Whittaker's functions $W_{\lambda, \mu}$, see \cite[Section 9.22]{GrRy}. Indeed, when $x>0$ and $y>0$, using \cite[p. 367]{GrRy}, we obtain :
\begin{align*}
h(x,y)&=  \frac{A^{-\nu}}{\Gamma\left(\nu\right)}  y^{2-\delta}\int_0^{x} (x-z)^{\beta}  z^{-\nu-1}\exp\left(- \frac{y^{2+\gamma} }{A\, z} \right) dz\\
&=  \frac{\Gamma(1+\beta)}{\Gamma(\nu)} A^{\frac{1-\nu}{2}} y^{- \frac{\delta+\gamma}{2}} x^{\frac{1-\nu}{2}+\beta}  e^{-\frac{y^{2+\gamma}}{2Ax}} W_{-\beta-\frac{1-\nu}{2} , \frac{\nu}{2}}\left(\frac{y^{2+\gamma}}{Ax} \right),
\end{align*}
while for $x>0$ and $y<0$, still from \cite[p. 368]{GrRy} : 
\begin{align*}
h(x,y)&= \frac{1}{\Gamma\left(\nu\right)}   \left(\frac{c}{A}\right)^{\nu} |y|^{2-\delta}  \int_{0}^{+\infty}(x+z)^{\beta}z^{-\nu-1}  \exp\left(- \frac{c|y|^{2+\gamma} }{A z} \right)  dz\\
&= \left( \frac{c}{A}\right)^{\frac{\nu-1}{2}} \frac{\Gamma(\nu-\beta)}{\Gamma(\nu)} x^{\beta+\frac{1-\nu}{2}} |y|^{- \frac{\delta+\gamma}{2}}  e^{\frac{c|y|^{2+\gamma}}{2A x}}\, W_{\beta+\frac{1-\nu}{2} , \frac{\nu}{2}}\left(\frac{c |y|^{2+\gamma}}{Ax}\right).
\end{align*}
In particular, letting  $x\rightarrow 0$, we deduce that for $y\leq 0$ :
$$h(0,y) = \frac{\Gamma(\nu-\beta)}{\Gamma(\nu)}  \left( \frac{c}{A}\right)^{\beta} |y|^{2\theta}.$$

\section{Proof of Theorem \ref{theo:A}}\label{sec:2}

\noindent
The proof is divided in three steps :
\begin{enumerate}[$i)$]
\item we first compute the Mellin transform of  $X_{\zeta_b}-b$ in Sections \ref{sec:2.1} and \ref{sec:MM}, 
\item we then deduce in Section \ref{sec:2.3} some crude estimates on the survival function of $\zeta_b$,
\item and we finally prove in Section \ref{sec:2.4} the asymptotics of Theorem \ref{theo:A}.
\end{enumerate}

\noindent

\subsection{Study of the first zero of $Y$ after $T_b$}\label{sec:2.1}

Define
\begin{equation}\label{eq:MM}
M_+ = \E\left[X_1^{-\frac{2}{2+\gamma}}\Un_{\{X_1>0,\, Y_1\leq0\}}\right] \qquad \text{ and }\qquad M_- = \E\left[|X_1|^{-\frac{2}{2+\gamma}}\Un_{\{X_1<0, \,Y_1\leq0\}}\right].
\end{equation}
The value of the ratio $M_-/M_+$ will be the key to compute the persistence exponent of $X$. We give its value in the following lemma, whose proof is postponed to the next Section  \ref{sec:MM}.
\begin{lemma}\label{lem:AA}
The moments $M_-$ and $M_+$ are finite and their ratio $M_-/M_+$ equals :
$$
\frac{M_-}{M_+}= \frac{c^{\nu} \frac{1-\eta}{1+\eta}\sin\left(\frac{2\pi}{2+\gamma}\right) +  \sin\left(\frac{\delta\pi}{2+\gamma}\right)}{\sin\left(\nu \pi\right)}.   
$$
\end{lemma}

\noindent
We start by computing the Mellin transform of $X_{\zeta_b}-b$.

\begin{proposition}\label{prop:Xzeta}
Let $\{x<b \text{ and }y\in \R\}$ or $\{x=b \text{ and }y<0\}$. The Mellin transform of $X_{\zeta_b}-b$ is given for $s\in \left(\beta-1, \beta\right)$ by :
\begin{equation}\label{eq:Xzeta}
\E_{(x,y)}\left[(X_{\zeta_b}-b)^{s}\right] =  \frac{\sin\left(\pi \beta\right)}{\sin\left(\pi(\beta- s)\right)}\E_{(x,y)}\left[(b-X_{\sigma_0})_+^{s}\right] + \E_{(x,y)}\left[(X_{\sigma_0}-b)_+^{s}\right].  
\end{equation}
As a consequence,
\begin{equation}\label{eq:MelYT}
 \E_{(x,y)}\left[ Y_{T_b}^{s}\right] =A^{\frac{s}{2+\gamma}} \frac{\Gamma(\nu)}{\Gamma\left(\nu-\frac{s}{2+\gamma}  \right)}  \E_{(x,y)}\left[(X_{\zeta_b}-b)^{\frac{s}{2+\gamma}}\right].
 \end{equation}
\end{proposition}

The Mellin transform of $X_{\zeta_b}-b$ may easily be inverted using Lemma \ref{lem:Xs0}. In particular, when $y=0$, then $\sigma_0=0$ and the random variable $X_{\zeta_b}-b$ follows a Beta distribution :
\begin{equation}\label{eq:Xzetab}
\Pb_{(x,0)}(X_{\zeta_b}-b\in dz) =\frac{\sin\left(\pi\beta\right)}{\pi}  (b-x)^{\beta} \frac{z^{-\beta}}{b-x+z}\,  dz.
\end{equation}

\begin{proof}
Observe first  that using the Markov property and the scaling property, we have 
\begin{equation}\label{eq:XYX}
 \E_{(x,y)}\left[(X_{\zeta_b}-b)^{\frac{s}{2+\gamma}}\right] =  \E_{(x,y)}\left[\E_{(0,Y_{T_b})}\left[X_{\sigma_0}^{\frac{s}{2+\gamma}}  \right]\right]=   \E_{(x,y)}\left[ Y_{T_b}^{s}\right] \times \E_{(0, 1)}\left[X_{\sigma_0}^{\frac{s}{2+\gamma}}\right].
\end{equation}
From Lemma \ref{lem:Xs0}, we deduce that the Mellin transform of $X_{\sigma_0}$ equals 
$$ \E_{(0,1)}\left[X_{\sigma_0}^{\frac{s}{2+\gamma}}\right] =A^{-\frac{s}{2+\gamma}}\frac{ \Gamma\left(\nu- \frac{s}{2+\gamma}\right)}{\Gamma\left(\nu\right)}$$
which yields the Mellin transform of $Y_{T_b}$ by (\ref{eq:XYX}). 
\noindent
It remains thus to compute the Mellin transform of $X_{\zeta_b}-b$.
Applying the strong Markov property and using the continuity of $V_\gamma(Y)$, we first write  for $z\geq0$ 
\begin{align*}
\Pb_{(x,y)}\left( X_t \geq b+z,\, Y_t\leq 0\right) &= \Pb_{(x,y)}\left( X_{\zeta_b} + \int_0^{t-\zeta_b}V_\gamma(Y_{u+\zeta_b})\,du\geq b+z,\, Y_t\leq 0,\, \zeta_b\leq t  \right)\\
&= \Pb_{(x,y)}\left( X_{\zeta_b} + \int_0^{t-\zeta_b} V_\gamma(\widehat{Y}_{u})du\geq b+z,\, \widehat{Y}_{t-\zeta_b}\leq 0,\, \zeta_b\leq t  \right)
\end{align*}
where $(\widehat{X},\widehat{Y})$ is an independent copy of $(X,Y)$ started from $(0,0)$. Integrating in $t$ and $z$, we deduce that for $r\in (\frac{2}{2+\gamma},1)$ :
$$\int_0^{+\infty} \E_{(x,y)}\left[ (X_t-b)^{-r}_+ \Un_{\{Y_t\leq 0\}}\right] dt= \int_0^{+\infty} \E_{(x,y)}\left[ (X_{\zeta_b}-b  + \widehat{X}_t)^{-r}_+ \Un_{\{\widehat{Y}_t\leq 0\}}\right] dt.
$$
The scaling property and the change of variable $t=u\,(X_{\zeta_b}-b)^{\frac{2}{2+\gamma}}$  then yield 
\begin{multline}\label{eq:Markovzeta}
\int_0^{+\infty} \E_{(x,y)}\left[ (t^{1+\frac{\gamma}{2}}X_1-b)^{-r}_+ \Un_{\{Y_1\leq 0\}}\right] dt\\= \E_{(x,y)}\left[(X_{\zeta_b}-b)^{\frac{2}{2+\gamma}-r} \right]  \int_0^{+\infty}   \E\left[ (1  + u^{1+\frac{\gamma}{2}}\widehat{X}_1)^{-r}_+ \Un_{\{\widehat{Y}_1\leq 0\}}\right] du.
\end{multline}
Now the integral on the right hand-side of (\ref{eq:Markovzeta}) may be computed by separating the two cases $\widehat{X}_1>0$ and $\widehat{X}_1< 0$.
On the one hand, using the change of variable $v=u\widehat{X}_1$ and Fubini-Tonelli's theorem, we obtain 
\begin{align}
\notag \int_0^{+\infty}   \E\left[ (1  + u^{1+\frac{\gamma}{2}}\widehat{X}_1)^{-r}_+ \Un_{\{\widehat{X}_1>0,\,\widehat{Y}_1\leq 0\}}\right] du
&= M_{+}  \int_0^{+\infty}  (1  + v^{1+\frac{\gamma}{2}})^{-r} dv\\
\label{eq:A+}&= \frac{2 M_{+}}{2+\gamma}   B\left(\frac{2}{2+\gamma}, r - \frac{2}{2+\gamma} \right)
\end{align}
where $B(x,y)$ denote the usual Beta function. On the other hand, we get similarly
\begin{align}
 \notag \int_0^{+\infty}   \E\left[ (1  + u^{1+\frac{\gamma}{2}}\widehat{X}_1)^{-r}_+ \Un_{\{\widehat{X}_1<0,\,\widehat{Y}_1\leq 0\}}\right] du
&= M_{-}  \int_0^{1}  (1  - v^{1+\frac{\gamma}{2}})^{-r} dv\\
\label{eq:A-}&= \frac{2M_{-}}{2+\gamma}   B\left(\frac{2}{2+\gamma},   1-r \right).
\end{align}
To compute the left hand-side of (\ref{eq:Markovzeta}), observe first that 
$$
 \int_0^{+\infty} \E_{(x,y)}\left[ (X_t-b)_+^{-r}\Un_{\{Y_t\leq 0\}}\right] dt = \E_{(x,y)}\left[ \int_{\sigma_0}^{+\infty} (X_t-b)_+^{-r}\Un_{\{Y_t\leq 0\}} dt \right].
$$
Indeed, if $t\leq \sigma_0$, we either have  $X_t\leq b$ when $y\leq 0$, or $Y_t\geq 0$ when $y\geq0$. Then, applying the Markov property and denoting as before  $(\widehat{X},\widehat{Y})$ an independent copy of $(X,Y)$ started from $(0,0)$~:
\begin{align}
\notag&  \E_{(x,y)}\left[ \int_{\sigma_0}^{+\infty} (X_t-b)_+^{-r}\Un_{\{Y_t\leq 0\}} dt \right]\\
 \notag&\qquad\quad =\E_{(x,y)}\left[ \int_{0}^{+\infty} (X_{\sigma_0} +\widehat{X}_t-b)_+^{-r}\Un_{\{\widehat{Y}_{t}\leq 0\}} dt \right]\\
\notag  &\qquad\quad= \E_{(x,y)}\left[(X_{\sigma_0}-b)_+^{\frac{2}{2+\gamma}-r}\right]  \int_{0}^{+\infty} \E\left[(1+u^{1+\frac{\gamma}{2}}\widehat{X}_1)_+^{-r}\Un_{\{\widehat{Y}_{1}\leq 0\}}  \right]du \\
\label{eq:Is0}&\hspace{3.5cm}+\E_{(x,y)}\left[(b-X_{\sigma_0})_+^{\frac{2}{2+\gamma}-r}\right]  \int_{0}^{+\infty} \E\left[(u^{1+\frac{\gamma}{2}} \widehat{X}_1-1)_+^{-r}\Un_{\{\widehat{Y}_{1}\leq 0\}}  \right]du
\end{align}
where the last equality  follows by scaling and the changes of variables $t = \pm u\, (X_{\sigma_0}-b)$. The first integral in (\ref{eq:Is0})  is the same one as in the right hand-side of (\ref{eq:Markovzeta}) while the second one equals 
\begin{align}
 \notag\int_0^{+\infty}   \E\left[ (u^{1+\frac{\gamma}{2}}\widehat{X}_1 -1)^{-r}_+ \Un_{\{\widehat{Y}_1\leq 0\}}\right] du
&=  M_{+}  \int_1^{+\infty}  (v^{1+\frac{\gamma}{2}}-1)^{-r} dv\\
\label{eq:A+-}&= \frac{2M_{+}}{2+\gamma }   B\left(r-\frac{2}{2+\gamma}, 1-r  \right).
\end{align}
Plugging relations (\ref{eq:Markovzeta}) - (\ref{eq:A+-}) together and setting $\frac{2}{2+\gamma} - r =  s$, we finally obtain :
$$
 \E_{(x,y)}\left[(X_{\zeta_b}-b)^{s} \right] \\
=  \frac{  B\left(-s, 1+s- \frac{2}{2+\gamma}  \right)  \E_{(x,y)}\left[(b-X_{\sigma_0})_+^{s}\right]  }{  B\left(\frac{2}{2+\gamma}, -s \right) +  \frac{M_{-}}{M_+}  B\left(\frac{2}{2+\gamma},   1+s- \frac{2}{2+\gamma} \right)}+ \E_{(x,y)}\left[(X_{\sigma_0}-b)_+^{s}\right].$$
The two classic identities for the Beta function $B$ and the Gamma function $\Gamma$, i.e.  
$$B(x,y) = \frac{\Gamma(x)\Gamma(y)}{\Gamma(x+y)} \qquad \text{ and }\qquad \Gamma(z)\Gamma(1-z) = \frac{\pi}{\sin(\pi z)},$$ 
yield then the simplification
$$ \E_{(x,y)}\left[(X_{\zeta_b}-b)^{s} \right] 
=    \frac{\sin\left(\frac{2\pi}{2+\gamma}\right) }{\sin\left(\frac{2\pi}{2+\gamma}-\pi s\right) - \frac{M_-}{M_+} \sin(\pi s )}\E_{(x,y)}\left[(b-X_{\sigma_0})_+^{s}\right]+ \E_{(x,y)}\left[(X_{\sigma_0}-b)_+^{s}\right]$$
and Proposition \ref{prop:Xzeta} now follows from Lemma \ref{lem:AA}.
\end{proof}


%
%
%

\subsection{Computation of $M_-$ and $M_+$}\label{sec:MM} It does not seem evident to evaluate $M_-$ and $M_+$ as we do not know in general the distribution of the pair $(X,Y)$. We may nevertheless compute these specific moments by relying on a special (complex) instance of the Feyman-Kac formula (see for instance \cite[Appendix C]{JanArea}). To do so, we first notice that 
$$
2 M_+ =  \E\left[ |X_1|^{-\frac{2}{2+\gamma}}\Un_{\{Y_1\leq0\}}\right] +  \E\left[ \sgn(X_1)  |X_1|^{-\frac{2}{2+\gamma}}\Un_{\{Y_1\leq0\}}\right] $$
and
$$
2 M_- =  \E\left[ |X_1|^{-\frac{2}{2+\gamma}}\Un_{\{Y_1\leq0\}}\right] -  \E\left[ \sgn(X_1)  |X_1|^{-\frac{2}{2+\gamma}}\Un_{\{Y_1\leq0\}}\right]. $$
Define the integral 
$$\Upsilon = \int_0^{+\infty} \E\left[e^{-\i X_t}\Un_{\{Y_t\leq0\}}\right] dt$$
Using the scaling property, the change of variable $s=t^{1+\gamma/2}$ and assuming that one can exchange the integral and the expectation, we obtain :
\begin{align*}
\mathfrak{Re}\left(\Upsilon\right) &=  \int_0^{+\infty} \E\left[\cos( t^{1+\frac{\gamma}{2}}X_1)\Un_{\{Y_1\leq0\}}\right] dt\\
&= \frac{2}{2+\gamma} \E\left[  \int_0^{+\infty} s^{\frac{2}{2+\gamma}-1}  \cos(s |X_1|) ds\,\Un_{\{Y_1\leq0\}}\right]\\
&=  \frac{2}{2+\gamma} \Gamma\left(\frac{2}{2+\gamma}\right)\cos\left( \frac{\pi}{2+\gamma}  \right) \E\left[ |X_1|^{-\frac{2}{2+\gamma}}\Un_{\{Y_1\leq0\}}\right] 
\end{align*}
and similarly 
$$\mathfrak{Im}\left(\Upsilon\right) =  -\frac{2}{2+\gamma} \Gamma\left(\frac{2}{2+\gamma}\right)\sin\left( \frac{\pi}{2+\gamma}  \right)\E\left[ \sgn(X_1)  |X_1|^{-\frac{2}{2+\gamma}}\Un_{\{Y_1\leq0\}}\right]. $$
Therefore, the evaluation of $M_-$ and $M_+$ is reduced to that of $\Upsilon$.
Now, to compute $\Upsilon$, let us consider the two following functions
$$\varphi(y) =   \begin{cases}
 \sqrt{y}K_{\nu}\left(2\nu(1+\i)y^\frac{1}{2\nu}\right) &\qquad \text{if }y\geq0\\
 \vspace{-.3cm}\\
c^{\frac{\nu}{2}}e^{-\i \frac{\nu\pi}{2}}\sqrt{|y|}K_{\nu}\left(2\nu(1-\i)\sqrt{c}|y|^{\frac{1}{2\nu}}\right)  + \Delta  \sqrt{|y|}I_{\nu}\left(2\nu(1-\i)\sqrt{c}|y|^{\frac{1}{2\nu}}\right)  &\qquad \text{if }y<0\
\end{cases}
$$
and
$$\psi(y) = \begin{cases}
 \sqrt{y}K_{\nu}\left(2\nu(1+\i)y^{\frac{1}{2\nu}}\right)+  c^{\frac{\nu}{2}} e^{-\i \frac{\nu\pi}{2}}   \frac{1-\eta}{1+\eta}\Delta  \sqrt{y}I_{\nu}\left(2\nu(1+\i)y^{\frac{1}{2\nu}}\right) &\qquad \text{if }y\geq0\\
 \vspace{-.3cm}\\
 c^{\frac{\nu}{2}} e^{-\i \frac{\nu\pi}{2}}\sqrt{|y|}K_{\nu}\left(2\nu(1-\i)\sqrt{c}|y|^{\frac{1}{2\nu}}\right)  &\qquad \text{if }y<0
\end{cases}
$$
where $K_\nu$ denotes the modified Bessel function of the second kind and the constant $\Delta$ equals 
$$\Delta  =\frac{\pi}{2\sin(\nu\pi)} \left(c^{\frac{\nu}{2}} e^{-\i \frac{\nu\pi}{2}}+c^{-\frac{\nu}{2}}e^{\i \frac{\nu\pi}{2}}\frac{1+\eta}{1-\eta}\right).$$

\noindent
Both functions are continuous on $\R$, twice differentiable in $\R\backslash\{0\}$ and are solutions of the differential equation
\begin{equation}\label{eq:FK}
\begin{cases}
\frac{1}{2}  \phi^{\prime\prime}(y) = \i  |y|^{\frac{1}{\nu}-2}(\Un_{\{y>0\}} - c\Un_{\{y<0\}} )   \phi(y) \quad \text{for }y\neq0,\\
\vspace{-.3cm}\\
(1+\eta)\phi^\prime(0^+) = (1-\eta)\phi^\prime(0^-).
\end{cases}
\end{equation}
Notice also that using the asymptotics (see \cite[Section 8.4]{GrRy}) :
\begin{equation}\label{eq:asymIK}
I_\nu(z) \,\equi_{z\rightarrow 0}\, \frac{z^\nu}{2^\nu \Gamma(\nu+1)}, \quad K_\nu(z) \,\equi_{z\rightarrow0}\, \frac{2^{\nu-1} \Gamma(\nu)}{z^\nu}  \quad \text{ and } \quad K_\nu(z) \mathop{\sim}\limits_{z\rightarrow+\infty} \sqrt{\frac{\pi}{2z}}e^{-z},
\end{equation}
we deduce that $\varphi(0)=\psi(0)$ as well as the limits
$$\lim_{y\rightarrow +\infty} \varphi(y) =0 \qquad \text{ and }\quad \lim_{y\rightarrow -\infty} \psi(y) =0.$$
Furthermore, since $\varphi$ and $\psi$  are solutions of (\ref{eq:FK}), their Wronskien 
$$W(y)=\varphi(y)\psi^\prime(y) - \psi(y)\varphi^\prime(y)$$
is such that $W^\prime(y)=0$ for any $y\neq0$. Therefore, $W$ is a step function, and we set 
$$\omega_- = W(0^-)=\frac{\pi}{4\nu\sin(\nu\pi)}\left(\frac{1+\eta}{1-\eta}+ c^{\nu}e^{-\i \nu\pi} \right)
\quad \text{ and }\quad \omega_+ = W(0^+) =  \frac{1-\eta}{1+\eta} \,\omega_-.$$

\begin{lemma}\label{lem:FK}
Let $f$ be a measurable  and bounded function  on $\R$ and define 
\begin{equation}\label{eq:phi}
\phi(z)=\varphi(z) \int_{-\infty}^z \psi(u) f(u)  |u|^{\frac{2\delta-2}{2-\delta}} du +\psi(z) \int_{z}^{+\infty} \varphi(u) f(u)  |u|^{\frac{2\delta-2}{2-\delta}} du. 
\end{equation}
Then for any $y\in \R$:
\begin{equation}\label{eq:FK}
\int_0^{+\infty}\E_{(0,y )}\left[\left(\omega_+ \Un_{\{Y_t>0\}} + \omega_- \Un_{\{Y_t<0\}} \right)  e^{ -\i  X_t } f( \emph{sgn}(Y_t) |Y_t|^{2-\delta})\right] dt= 2\, \phi(\emph{sgn}(y) |y|^{2-\delta}). \end{equation}
\end{lemma}
\begin{proof} Notice first that the function $\phi$ is continuous and such that $\lim\limits_{|z|\rightarrow +\infty} \phi(z)=0.$ Indeed, for $u>0$, using integral representations for $I_\nu$ and $K_\nu$ (see for instance \cite[Section 8.43]{GrRy}), we have  :
$$\left|K_{\nu}\left(2\nu(1+\i)u^{\frac{1}{2\nu}}\right)\right| = \left|\int_0^{+\infty} e^{-2\nu(1+\i) u^{\frac{1}{2\nu}}\cosh(t)} \cosh\left(\nu t\right) dt\right| \leq K_{\nu}\left(2\nu u^{\frac{1}{2\nu}}\right)$$
and
$$\left|I_{\nu}\left(2\nu(1+\i)u^{\frac{1}{2\nu}}\right)\right| =  \left|\frac{ \sqrt{u}\left(\nu(1+\i)\right)^{\nu}}{\Gamma(\nu+\frac{1}{2}) \sqrt{\pi}} \int_{0}^{\pi}  e^{- \left(2\nu(1+\i)u^{\frac{1}{2\nu}}\right) \cos(t)}  \left(\sin(t)\right)^{2\nu} dt \right|\leq  2^{\frac{\nu}{2}} I_{\nu}\left(2\nu u^{\frac{1}{2\nu}}\right)$$
hence, from (\ref{eq:asymIK}), we deduce that there exist two constants $\kappa_1, \kappa_2>0$ such that for $z\geq0$ :
\begin{multline*}
\left|\varphi(z) \int_{-\infty}^z \psi(u) f(u)  |u|^{\frac{2\delta-2}{2-\delta}} du\right| \\ \leq  \kappa_1  \sup_{x\in \R}|f(x)| \sqrt{z}  K_{\nu}\left(2\nu z^{\frac{1}{2\nu}}\right)  \left( \kappa_2 + \int_0^{z} \sqrt{u}I_{\nu}\left(2\nu u^{\frac{1}{2\nu}}\right) |u|^{\frac{2\delta-2}{2-\delta}}  du   \right).
\end{multline*}
Recalling the asymptotics 
$
\displaystyle I_\nu(z) \mathop{\sim}\limits_{z\rightarrow+\infty} \frac{e^z}{\sqrt{2\pi z}}$ and applying Watson's lemma
we deduce that there exists $\kappa>0$ such that for $z$ large enough :
$$\left|\varphi(z) \int_{-\infty}^z \psi(u) f(u)  |u|^{\frac{2\delta-2}{2-\delta}} du\right| \leq \kappa z^{- \frac{\gamma}{2-\delta}},\qquad\qquad z\rightarrow +\infty.$$
Proceeding similarly for the second integral in (\ref{eq:phi}), we conclude that $\lim\limits_{z\rightarrow +\infty} \phi(z)=0$, and by similar arguments that we also have $\lim\limits_{z\rightarrow -\infty} \phi(z)=0$.
Next, differentiating $\phi$  for $z\neq 0$, we obtain 
$$\phi^\prime(z)=\varphi^\prime(z) \int_{-\infty}^z \psi(u) f(u) |u|^{\frac{2\delta-2}{2-\delta}}  du +\psi^\prime(z) \int_{z}^{+\infty} \varphi(u) f(u)  |u|^{\frac{2\delta-2}{2-\delta}} du. $$
It thus follows from (\ref{eq:FK}) that 
\begin{equation}\label{eq:phi''}
\frac{1}{2}\phi^{\prime\prime}(dz)  = \i  V_{\frac{1}{\nu}-2}(z) \phi(z)dz - \frac{1}{2}W(z)f(z) |z|^{\frac{2\delta-2}{2-\delta}} dz + (\phi^\prime(0^+) -\phi^\prime(0^-)) \delta_0(dz)\end{equation}
where $ \delta_0$ denotes the Dirac measure at 0. Let us set
$$Z_t = \sgn(Y_t) |Y_t|^{2-\delta}.$$
\noindent
From Blei \cite[Theorem 2.22]{Ble}, $Z$ is the unique strong solution of the SDE :
\begin{equation}\label{eq:Blei}
Z_t = \text{sgn}(y)|y|^{2-\delta}+(2-\delta) \int_0^t |Z_s|^{\frac{1-\delta}{2-\delta}} dB_s + \eta L_t^{0}(Z)
\end{equation}
where $L_t^{0}(Z)$ denotes the semimartigale symmetric local time of $Z$.
Consider now the process
\begin{equation}\label{eq:defM}
M_t= e^{- \i \int_0^t V_{\frac{\gamma}{2-\delta}}(Z_s) ds} \phi(Z_t) + \frac{1}{2}\int_0^t  W(Z_u) e^{- \i \int_0^uV_{\frac{\gamma}{2-\delta}}(Z_s) ds} f(Z_u)du.
\end{equation}
We apply the symmetric It\^o-Tanaka formula (see \cite[Section 5.1]{Lej}) setting $\phi^\prime_{r}$ (resp. $\phi^\prime_{\ell}$) for the right-derivative (resp. the left-derivative) of $\phi$ :
\begin{align*}
dM_t &= \i  V_{\frac{\gamma}{2-\delta}}(Z_t) e^{- \i \int_0^t V_{\frac{\gamma}{2-\delta}}(Z_s) ds} \phi(Z_t)dt +    \frac{1}{2}\ e^{- \i \int_0^t V_{\frac{\gamma}{2-\delta}}(Z_s) ds} (\phi_{r}^\prime(Z_t) + \phi_{\ell}^\prime(Z_t)) dZ_t \\
&\qquad \qquad + \frac{1}{2} e^{- \i \int_0^t V_{\frac{\gamma}{2-\delta}}(Z_s) ds}    \int_{-\infty}^{+\infty}\phi^{\prime\prime}(da) dL_t^a(Z)  + \frac{1}{2}W(Z_t)  e^{- \i \int_0^tV_{\frac{\gamma}{2-\delta}}(Z_s) ds} f(Z_t).
\end{align*}
Using (\ref{eq:phi''}), (\ref{eq:Blei}) and the occupation time formula, we then obtain 
$$dM_t = \frac{2-\delta }{2}  e^{- \i \int_0^t V_{\frac{\gamma}{2-\delta}}(Z_s) ds} (\phi_r^\prime(Z_t) + \phi_{\ell}^\prime(Z_t)) |Z_t|^{\frac{1-\delta}{2-\delta}}  dB_t$$ 
which proves that $M$ is a local martingale. Furthermore, going back to the definition (\ref{eq:defM})  of $M$, we have the estimate for $r\geq0$ :
$$\sup_{0\leq t\leq r}| M_t | \leq \sup_{z\in \R}|\phi(z)| +r \frac{\max(|\omega_+|, |\omega_-|)}{2} \sup_{z\in \R}|f(z)| . $$
Therefore $M$ is bounded on $[0,r]$ which implies that $M$ is a martingale and the equality $\E_{(0,y)}\left[M_0\right]=\E_{(0,y)}\left[M_t\right]$ yields
$$
\phi(\sgn(y) |y|^{2-\delta})=\frac{1}{2} \E_{(0,y)}\left[\int_0^{t}  W(Z_u)e^{-\i \int_0^uV_{\frac{\gamma}{2-\delta}}(Z_s) ds} f(Z_u)du\right]+\E_{(0, y)}\left[e^{- \i\int_0^t V_{\frac{\gamma}{2-\delta}}(Z_s) ds} \phi(Z_t)\right].
$$
Lemma \ref{lem:FK} now follows by letting $t\rightarrow +\infty$, applying Fubini's theorem on the first expectation, and the scaling property and the dominated convergence theorem on the second since $\lim\limits_{|z|\rightarrow +\infty}\phi(z)=0$.

\end{proof}

To compute $M_-$ and $M_+$, we shall now apply Lemma \ref{lem:FK} with $y=0$ and $f(u) = \Un_{\{u\leq0\}}$. Using \cite[p. 676, Formula 16]{GrRy}, we obtain :
\begin{align}\label{eq:AppFK}
\int_0^{+\infty} \E\left[  e^{ -\i t^{1+\frac{\gamma}{2}}X_1}\Un_{\{Y_1\leq 0\}}\right] dt  &=2  \frac{\varphi(0)}{\omega_-} \int_{-\infty}^{0} \psi(y) |y|^{\frac{2\delta-2}{2-\delta}} dy \\
&\notag =H(\delta, \eta, \gamma, c) \times \left(c^{\nu}  e^{\i \frac{\pi}{2+\gamma}} + \frac{1+\eta}{1-\eta} e^{\i \frac{\pi(\delta-1)}{2+\gamma}}\right)
\end{align}
where $H(\delta, \eta, \gamma, c)$ is a positive constant, and it remains to prove that we may exchange the integral and the expectation on the left hand-side of (\ref{eq:AppFK}). To do so, it is sufficient by Lemma 1 in \cite{PSPers} to prove that $ \E\left[ |X_1|^{-\frac{2}{2+\gamma}}\Un_{\{Y_1\leq 0\}}\right]<+\infty$. To this end, to simplify the notation, let us define for $s>0$ 
$$g(s) = s^{\frac{2}{2+\gamma}-1} \E\left[\cos(s|X_1|)\Un_{\{Y_1\leq0\}}\right]$$
and  for $\lambda\geq 0$ 
$$G(\lambda) = \int_0^{+\infty} e^{-\lambda s} g(s)ds.$$
Integrating twice by parts, we deduce that
$$G(\lambda) = \int_0^{+\infty} \E\left[ \frac{1-\cos(s|X_1|)}{X_1^2}  \right] s^{\frac{2}{2+\gamma}-3}e^{-\lambda s} \left(\lambda^2 s^2 + \frac{2\lambda \gamma}{2+\gamma}   s+ \frac{2\gamma(1+\gamma)}{(2+\gamma)^2}\right)ds. $$
Then, applying Fatou's lemma and the Fubini-Tonelli theorem, we obtain 
$$\lim_{\lambda \rightarrow 0^+} G(\lambda) \geq \Gamma\left(\frac{2}{2+\gamma}\right)\cos\left( \frac{\pi}{2+\gamma}  \right) \E\left[ |X_1|^{-\frac{2}{2+\gamma}}\Un_{\{Y_1\leq0\}}\right]$$
and it remains to prove that the limit on the left hand-side is finite. Let $\varepsilon>0$. Since $G(0)$ is finite from (\ref{eq:AppFK}), we may choose $R$ large enough such that
$$\forall u\geq R, \qquad \left|\int_u^{+\infty} g(s)ds  \right|<\varepsilon.$$
We then decompose
$$
|G(0)-G(\lambda)| \leq \left| \int_0^{R} (1-e^{-\lambda s}) g(s)  ds\right| + \left| \int_R^{+\infty} (1-e^{-\lambda s}) g(s) ds\right|. 
$$
The first term on the right hand-side goes to zero as $\lambda\rightarrow0^+$ by the dominated convergence theorem while integrating by parts, the second term equals :
$$
\left|\left(1-e^{-\lambda R}\right) \int_R^{+\infty} g(s)ds + \lambda\int_R^{+\infty} e^{-\lambda u} \left(  \int_u^{+\infty}  g(s)ds\right) du\right|
\leq \varepsilon + \varepsilon \lambda \int_R^{+\infty} e^{-\lambda u} du \leq 2\varepsilon.
$$
This implies that $G$ is continuous at $0^+$, hence
$$ \Gamma\left(\frac{2}{2+\gamma}\right)\cos\left( \frac{\pi}{2+\gamma}  \right) \E\left[ |X_1|^{-\frac{2}{2+\gamma}}\Un_{\{Y_1\leq0\}}\right] \leq G(0)<+\infty.$$

\noindent
Finally, taking the real and imaginary parts in (\ref{eq:AppFK}), we deduce that :
$$
\frac{M_-}{M_+}= \frac{ 2 c^{\nu} +\frac{1+\eta}{1-\eta} \left(\frac{\cos\left(\frac{\pi (\delta-1)}{2+\gamma}\right) }{\cos\left(\frac{\pi}{2+\gamma}\right)} + \frac{\sin\left(\frac{\pi (\delta-1)}{2+\gamma}\right) }{\sin\left(\frac{\pi}{2+\gamma}\right)} \right)  }{\frac{1+\eta}{1-\eta} \left(\frac{\cos\left(\frac{\pi (\delta-1)}{2+\gamma}\right) }{\cos\left(\frac{\pi}{2+\gamma}\right)} - \frac{\sin\left(\frac{\pi (\delta-1)}{2+\gamma}\right) }{\sin\left(\frac{\pi}{2+\gamma}\right)} \right)  }
= \frac{c^{\nu} \frac{1-\eta}{1+\eta}\sin\left(\frac{2\pi}{2+\gamma}\right) +  \sin\left(\frac{\delta\pi}{2+\gamma}\right)}{\sin\left(\nu \pi\right)}
$$
which ends the proof of Lemma \ref{lem:AA}.

\qed

\begin{remark}
One may observe that the ratio $M_-/M_+$ admits a finite limit when $c\rightarrow 0^+$. However, when $c=0$, the process $X$  is always non negative, and one can check that  both moments $M_-$ and $M_+$ are no longer finite, see for instance Janson \cite[Section 29]{JanArea} in the Brownian case. \\
\end{remark}

\subsection{The asymptotics of $\zeta_b$}\label{sec:2.3}

We momentarily assume that $x=y=0$ and first prove the following crude asymptotics.
\begin{lemma}\label{lem:zetasup}
There exist two constants $0<\kappa_1\leq \kappa_2 < +\infty$ such that 
$$\kappa_1\, t^{-\theta}\leq \Pb(\zeta_b\geq t) \leq \kappa_2\, t^{-\theta}, \qquad t\rightarrow +\infty.$$
\end{lemma}

\begin{proof}
The lower bound is easy to obtain. Indeed, from the Markov property and the scaling property, we have 
$$\zeta_b \= T_b + Y_{T_b}^2 \,\tau_1$$
where on the right hand-side, $\tau_1$ is independent from the pair $(T_b, Y_{T_b})$ and has the same law as $\sigma_0$ under $\Pb_{(0,1)}$, which is given (see \cite[Equation (13)]{GoYo}) by
\begin{equation}\label{eq:sigma0}
\Pb_{(0,1)}\left(\sigma_0\in dt \right) = \frac{2^{\frac{\delta}{2}-1}}{\Gamma\left(1-\frac{\delta}{2}\right)} t^{\frac{\delta}{2}-2} e^{-\frac{1}{2t}}dt.
\end{equation}
This yields the lower bound 
\begin{equation}\label{eq:mintau}
\Pb(\zeta_b \geq t)\; \geq \;\Pb(Y_{T_b}^2  \, \tau_1 \geq t).\end{equation}
From (\ref{eq:sigma0}),  the asymptotics of $\tau_1$ equals 
\begin{equation}\label{eq:asymps0}
\Pb(\tau_1\geq t) = \Pb_{(0,1)}(\sigma_0\geq t) \equi_{t\rightarrow +\infty}  \frac{2^{\frac{\delta}{2}-1}}{\Gamma\left(2-\frac{\delta}{2}\right)} t^{\frac{\delta}{2}-1}
\end{equation}
while, using  (\ref{eq:MelYT}) and the converse mapping for Mellin transform (see \cite[Theorem 4]{FGD}), that of $Y_{T_b}^2$ equals  :
$$
 \Pb(Y_{T_b}^2 \geq z) \equi_{z\rightarrow +\infty}  \kappa z^{-\theta}$$
for some constant $\kappa>0$. Finally, since $\theta< 1-\frac{\delta}{2}$, the leading term in the product on the right hand-side of (\ref{eq:mintau}) is $Y_{T_b}^2$, and the lower bound follows by applying Lemma 2.2 in \cite{PSAlea}.\\

\noindent
The upper bound is much more involved, and we shall follow the idea of Profeta-Simon \cite{PSPers}. Applying the Markov property and Fubini's theorem, we have for $\lambda>0$ and taking $r\in \left(0,\frac{\delta}{2+\gamma}\right)$:
$$\int_0^{\infty} e^{-\lambda t}\,\E\left[(X_t-b)_+^{-r}\Un_{\{Y_t\leq0\}}\right]dt
\;=\;\E\left[e^{-\lambda \zeta_b}  \int_0^{\infty} e^{-\lambda t}\,\E_{(X_{\zeta_b},0)}\left[(X_t-b)_{+}^{-r}\Un_{\{Y_t\leq0\}}\right]dt \right].$$
Integrating by parts, we then deduce 
\begin{align*}
&\lambda \int_0^{\infty} e^{-\lambda t}\int_t^{\infty} \left(\E\left[ (X_u-b)_+^{-r}\Un_{\{Y_u\leq0\}}\right]- \E\left[ \E_{(X_{\zeta_b,0})}\left[ (X_u-b)_+^{-r}\Un_{\{Y_u\leq0\}}\right] \right] \right)du\, dt\\
&\qquad= \E\left[  (1-e^{-\lambda \zeta_b})\int_0^{\infty} e^{-\lambda t}\E_{(X_{\zeta_b,0})}\left[ (X_t-b)_+^{-r}\Un_{\{Y_t\leq0\}}\right] dt \right]\\
&\qquad =\E\left[ \int_0^{\infty} \lambda \,e^{-\lambda t}\left( \int_0^t \Un_{\{\zeta_b>t-u\}} \E_{(X_{\zeta_b,0})}\left[ (X_u-b)_+^{-r}\Un_{\{Y_u\leq0\}}\right] du\right) dt \right].
\end{align*}
Inverting the Laplace transforms shows that
\begin{equation}\label{eq:invLap}
\E\left[ \int_0^t \Un_{\{\zeta_b>t-u\}}\E_{(X_{\zeta_b,0})}\left[ (X_u-b)_+^{-r}\Un_{\{Y_u\leq0\}}\right]  du \right] \; = \; H(t)
\end{equation}
with the notation
\begin{equation}
\label{Hxy}
H(t)\;=\;\int_t^{+\infty} \left(\E\left[ (X_u-b)_+^{-r}\Un_{\{Y_u\leq0\}}\right]- \E\left[ \E_{(X_{\zeta_b,0})}\left[ (X_u-b)_+^{-r}\Un_{\{Y_u\leq0\}}\right] \right] \right)du, \qquad t >0.
\end{equation}
To get the asymptotics of $H$, we shall compute the Mellin transform of its derivative. Proceeding as for (\ref{eq:Markovzeta}) and applying Proposition \ref{prop:Xzeta}, we have for $0<s<r$:
$$
\int_0^{+\infty} t^{\frac{2+\gamma}{2}s-1} H^\prime(t) dt =\left(1+\frac{\gamma}{2}\right)\left( E_1 - E_2\right)$$
where
$$E_1 = \Gamma(r-s)b^{s -r} \E\left[X_1^{-s}\Un_{\{X_1>0,\, Y_1\leq0\}}\right] \left(\frac{\Gamma(1-r)}{\Gamma(1-s)} -   \frac{\sin\left(\pi \beta\right)}{\sin\left(\pi(\beta- s+r)\right)} \frac{\Gamma(s)}{\Gamma(r)} \right) $$
and
$$E_2 = b^{s -r}  \frac{\sin\left(\pi \beta\right)}{\sin\left(\pi(\beta- s+r)\right)}\E\left[|X_1|^{-s}\Un_{\{X_1<0,\, Y_1\leq0\}}\right] B\left(s, 1-r\right). $$
This identity may be extended by analytic continuation to $s\in (0, r + \frac{2\theta}{2+\gamma})$, using that $r+  \frac{2\theta}{2+\gamma} < \frac{2}{2+\gamma}$ to ensure that the negative moments of $X_1$ are finite from Lemma \ref{lem:AA}. Since the pole  at $r +\frac{2\theta}{2+\gamma}$ is simple, we deduce from the converse mapping for Mellin transform, upon integration, that  
\begin{equation}\label{eq:asympH}
H(t) \equi_{t\rightarrow +\infty}  \kappa\, t^{1-\frac{2+\gamma}{2}r - \theta}    
\end{equation}
for some constant $\kappa>0$. Then, going back to  (\ref{eq:invLap}), applying the Markov property and denoting as before $(\widehat{X}, \widehat{Y})$ an independent copy of $(X,Y)$ started at $(0,0)$, we have 
\begin{align*}
\notag t^{\frac{2+\gamma}{2}r+ \theta-1} H (t) & \geq \; t^{\frac{2+\gamma}{2}r+ \theta-1}\,\E\left[ \Un_{\{\zeta_b>t\}}\int_0^t  \E_{(X_{\zeta_b},0)}\left[ (X_u-b)_+^{-r}\Un_{\{Y_u\leq0\}}\right]  du\right] \\
\notag & \geq \;t^{\frac{2+\gamma}{2}r+ \theta}\,\E\left[ \Un_{\{\zeta_b>t\}}\int_0^1  (X_{\zeta_b}-b + (tv)^{1+\frac{\gamma}{2}} \widehat{X}_1)_+^{-r}  \Un_{\{\widehat{X}_1>0, \, \widehat{Y}_1\leq 0\}}  \Un_{\{X_{\zeta_b}-b \leq  t^{1+\frac{\gamma}{2}}\}}dv \right] \\ 
&\geq  \;t^{\theta}\, \Pb\left(\zeta_b>t ,X_{\zeta_b}-b \leq  t^{1+\frac{\gamma}{2}} \right)\underbrace{\E\left[\int_0^1  ( 1 + v^{1+\frac{\gamma}{2}} \widehat{X}_1)^{-r}  \Un_{\{\widehat{X}_1>0, \, \widehat{Y}_1\leq 0\}} dv \right]}_{=k}\\
&=   k t^{\theta} \left( \Pb\left(\zeta_b>t \right) -  \Pb\left(\zeta_b>t, \,X_{\zeta_b}-b \geq  t^{1+\frac{\gamma}{2}} \right)\right).
\end{align*}
We finally obtain 
$$ t^{\frac{2+\gamma}{2}r+ \theta-1} H (t) + k t^{\theta}\, \Pb\left(X_{\zeta_b}-b \geq  t^{1+\frac{\gamma}{2}} \right) \geq  k t^{\theta}\, \Pb\left(\zeta_b>t \right)$$
and the upper bound of Lemma \ref{lem:zetasup} follows from (\ref{eq:asympH}) and (\ref{eq:Xzetab}).

%
%
%

\end{proof}

\subsection{Proof of Theorem \ref{theo:A}}\label{sec:2.4}
We first prove that $(x,y)\mapsto h(-x,y)$ is harmonic for the killed process. Indeed, observe first that from Proposition \ref{prop:Xzeta}, there exists a constant $\kappa$ independent from $(x,y)$ such that for $\{x<0\text{ and }y\in \R\}$ or $\{x=0 \text{ and }y<0\}$ :
$$h(-x,y) = \kappa \lim_{s\rightarrow \theta} \left( s-\theta \right) \E_{(x,y)}\left[ Y_{T_0}^{2s}\right]. $$
Then, applying the Markov property, we deduce that for any $t>0$, 
\begin{align*}
\E_{(x,y)}\left[h(-X_t, Y_t) \Un_{\{t< T_0\}}\right] &= \kappa\lim_{s\rightarrow\theta} (s-\theta) \E_{(x,y)}\left[ \E_{(X_t,Y_t)}\left[ Y_{T_0}^{2s}\right] \Un_{\{t< T_0\}}\right] \\
&= \kappa\lim_{s\rightarrow\theta} (s-\theta) \E_{(x,y)}\left[ Y_{T_0}^{2s}\Un_{\{t< T_0\}}\right] \\
&=h(-x,y) - \kappa\lim_{s\rightarrow\theta} (s-\theta)\E_{(x,y)}\left[ Y_{T_0}^{2s}\Un_{\{t\geq T_0\}}\right] \\
&=h(-x,y)
\end{align*}
since 
$$0\leq \E_{(x,y)}\left[ Y_{T_0}^{2s}\Un_{\{t\geq T_0\}}\right]\leq  \E_{(x,y)}\left[ \left(1\wedge Y_{T_0}\right)^{2\theta}\Un_{\{t\geq T_0\}}\right]  \leq  \E_{(x,y)}\left[ \left(1\wedge\sup_{u\leq t }Y_{u}\right)^{2\theta}\right]<+\infty.$$
This allows to define a new probability measure by the absolute continuity formula 
$$\Q_{(x,y)|\F_t} = \frac{h(-X_t,Y_t)}{h(x,y)} \Un_{\{t<T_0\}}\centerdot \Pb_{(x,y)|\F_t}$$
where  $(\F_t)_{t\geq0}$ denotes the natural filtration of the process $(X,Y)$. The measure $\Q$ is in fact the law $(X,Y)$ where $X$ is conditioned to remain negative. By translation and scaling, we deduce using (\ref{eq:hscale}) that
\begin{align*}
\Pb_{(x,0)}(T_b>t) &= (b-x)^{\frac{2\theta}{2+\gamma}}\Q_{(0,0)}\left[\frac{1}{h(b-x-X_t, Y_t)}  \right]\\ 
&= (b-x)^{\frac{2\theta}{2+\gamma}} t^{-\theta} \Q_{\left(0,0\right)}\left[\frac{1}{h\left( \frac{b-x}{t^{1+\frac{\gamma}{2}}}-X_1, Y_1\right)}  \right]. 
\end{align*}
Since $h$ is increasing in its first variable, we deduce from the monotone convergence theorem that   
$$\lim_{t\rightarrow+\infty}t^{\theta}\Pb_{(x,0)}(T_b>t) =(b-x)^{\frac{2\theta}{2+\gamma}} \Q_{\left(0,0\right)}\left[\frac{1}{h(-X_1, Y_1)}  \right]$$
and this last quantity is finite since from Lemma \ref{lem:zetasup}
$$t^{\theta}\Pb_{(x,0)}(T_b>t) \leq t^{\theta}\Pb_{(x,0)}(\zeta_b>t) \leq \kappa_2\qquad \text{ as } t\rightarrow +\infty.$$
Finally, when starting from $\{x<b \text{ and } y\neq0\}$ or $\{x=b\text{ and } y<0\}$, we have applying the Markov property
\begin{align*}
t^{\theta} \Pb_{(x,y)}(T_b>t) &=  t^{\theta} \Pb_{(x,y)}(T_b>t ,\, \sigma_0\leq T_b) +  t^{\theta}\Pb_{(x,y)}(T_b>t ,\, \sigma_0> T_b) \\
 &=   t^{\theta}\E_{(x,y)}\left[Ê\Pb_{(X_{\sigma_0},0)}\left(s +  T_b >t \right)_{| s=\sigma_0}  \Un_{\{\sigma_0\leq T_b\}} \Un_{\{X_{\sigma_0}\leq b\}}\right]+t^{\theta} \Pb_{(x,y)}\left( \sigma_0> T_b>t\right)\\
 &\xrightarrow[t\rightarrow +\infty]{} \E_{(x,y)}\left[(b-X_{\sigma_0})_+^{\frac{2\theta}{2+\gamma}}\right]  \Q_{\left(0,0\right)}\left[\frac{1}{h(-X_1, Y_1)}  \right],
\end{align*}
since from (\ref{eq:asymps0}), recalling that $\theta < 1-\frac{\delta}{2}$, 
$$\lim_{t\rightarrow+\infty} t^{\theta} \Pb_{(x,y)}\left( \sigma_0>t\right)=0.$$
\qed

\section{The exit time problem}\label{sec:3}

We now prove Theorems \ref{theo:B}, \ref{theo:C} and Corollary \ref{cor:D}. Recall for $a\leq 0 \leq b$ the definitions 
$$\zeta_a = \inf\{t>0, \, X_t\leq a \text{ and }Y_t=0\}, \qquad \zeta_b = \inf\{t>0, \, X_t\geq b \text{ and }Y_t=0\}$$
and
$$\zeta_{ab} =\zeta_a\wedge \zeta_b = \inf\{t>0,\, X_t \neq(a,b) \text{ and }Y_t=0\}.$$
We start by studying the random variable $X_{\zeta_{ab}}$. We shall proceed as in Section \ref{sec:2} but write this time a system of equations since $X_{\zeta_{ab}}$ take values in $\R\backslash(a,b)$.

\subsection{The distribution of $X_{\zeta_{ab}}$} Let $x\in(a,b)$.
Applying the strong Markov property and using the continuity of $V_\gamma(Y)$, we deduce that 
\begin{align*}
\Pb_{(x,0)}\left( X_t \geq b+z,\, Y_t\leq 0\right) &= \Pb_{(x,0)}\left( X_{\zeta_b} + \int_0^{t-\zeta_b}V_\gamma(\widehat{Y}_{u})du\geq b+z,\, Y_t\leq 0,\, \zeta_{ab}\leq t, \zeta_b< \zeta_a  \right) \\
&\qquad +  \Pb_{(x,0)}\left( X_{\zeta_a} + \int_0^{t-\zeta_a}V_\gamma(\widehat{Y}_{u})du\geq b+z,\, Y_t\leq 0,\, \zeta_{ab}\leq t, \zeta_a< \zeta_b  \right)
\end{align*}
where $(\widehat{X}, \widehat{Y})$ is an independent copy of $(X,Y)$ started from $(0,0)$.
Integrating in $z$ and $t$, we obtain for $r\in (\frac{2}{2+\gamma}, 1)$ :
\begin{multline*}
\int_0^{+\infty} \E_{(x,0)}\left[ (X_t-b)^{-r}_+ \Un_{\{Y_t\leq 0\}}\right] dt= \int_0^{+\infty} \E_{(x,0)}\left[ (X_{\zeta_b}-b  + \widehat{X}_t)^{-r}_+ \Un_{\{\widehat{Y}_t\leq 0\}}\Un_{\{\zeta_b< \zeta_a \}} \right] dt\\ +  \int_0^{+\infty} \E_{(x,0)}\left[ (X_{\zeta_a}-b  + \widehat{X}_t)^{-r}_+ \Un_{\{\widehat{Y}_t\leq 0\}}\Un_{\{\zeta_a< \zeta_b \}} \right] dt 
\end{multline*}
From (\ref{eq:A+-}), the left hand-side equals 
$$  \int_0^{+\infty}   \E\left[ (x+t^{1+\frac{\gamma}{2}}X_1 -b)^{-r}_+ \Un_{\{Y_1\leq 0\}}\right] dt=\frac{2M_+}{2+\gamma}  (b-x)^{\frac{2}{2+\gamma}-r}  B\left(r-\frac{2}{2+\gamma}, 1-r  \right).$$
Separating the cases $\widehat{X}_1>0$ and $\widehat{X}_1<0$ and proceeding as in (\ref{eq:A+}) and (\ref{eq:A-}), we obtain 
\begin{multline*}
 \int_0^{+\infty} \E\left[ (X_{\zeta_b}-b  +\widehat{X}_t)^{-r}_+ \Un_{\{\widehat{Y}_t\leq 0\}}\Un_{\{\zeta_b< \zeta_a \}} \right] dt \\
=  \E\left[(X_{\zeta_{ab}}-b)^{\frac{2}{2+\gamma}-r}\Un_{\{X_{\zeta_{ab}>0} \}}  \right] \left( \frac{2M_{+}}{2+\gamma}   B\left(\frac{2}{2+\gamma}, r - \frac{2}{2+\gamma} \right) +  \frac{2M_{-}}{2+\gamma}   B\left(\frac{2}{2+\gamma},   1-r \right)\right)
\end{multline*}
and
\begin{multline*}
  \int_0^{+\infty} \E\left[ (X_{\zeta_a}-b  + \widehat{X}_t)^{-r}_+ \Un_{\{\widehat{Y}_t\leq 0\}}\Un_{\{\zeta_a< \zeta_b \}} \right] dt \\
=\E\left[(b-X_{\zeta_{ab}})^{\frac{2}{2+\gamma}-r}\Un_{\{X_{\zeta_{ab}}< 0 \}}  \right] \frac{2M_{+}}{2+\gamma}   B\left(r-\frac{2}{2+\gamma}, 1-r  \right)
\end{multline*}
where we have used that $\{\zeta_b<\zeta_a\} \Longleftrightarrow \{X_{\zeta_{ab}}>0\}$ a.s.
Setting $s=\frac{2}{2+\gamma}-r$ and using Lemma \ref{lem:AA}, we thus obtain the equation :
$$ \E\left[\left(X_{\zeta_{ab}}-b\right)^{s}\Un_{\{X_{\zeta_{ab}>0} \}}\right]\frac{\sin(\pi (\beta-s) )}{\sin(\pi \beta)} +   \E\left[\left(b-X_{\zeta_{ab}}\right)^{s}\Un_{\{X_{\zeta_{ab}<0} \}}\right]   = (b-x)^s.$$
Similarly, we obtain by symmetrical arguments
$$ \E\left[\left(a-X_{\zeta_{ab}}\right)^{s}\Un_{\{X_{\zeta_{ab}<0} \}}\right] \frac{\sin(\pi (\alpha-s) )}{\sin(\pi \alpha)}+   \E\left[\left(X_{\zeta_{ab}}-a\right)^{s}\Un_{\{X_{\zeta_{ab}>0} \}}\right]   = (x-a)^s.$$
Notice that in this system, the dependence in the parameters $\delta, \eta, \gamma$ and $c$ only appears through $\alpha$ and $\beta$. 
\begin{lemma}\label{lem:lawXzetaab} We set $J=(\max(\alpha,\beta)-1, \min(\alpha,\beta))$. The system of integral equations :
\begin{equation}\label{eq:sys}
\left\{\begin{array}{l}
\displaystyle \frac{\sin(\pi (\beta-s) )}{\sin(\pi \beta)} \int_0^{+\infty} z^s \mu_+(dz) + \int_0^{+\infty} (b-a+z)^s \mu_-(dz) = (b-x)^s\\
\vspace{-.3cm}\\
\displaystyle \frac{\sin(\pi (\alpha-s) )}{\sin(\pi \alpha)} \int_0^{+\infty} z^s \mu_-(dz) + \int_0^{+\infty} (b-a+z)^s \mu_+(dz) = (x-a)^s\\
\end{array}\right.,
\qquad s\in J,
\end{equation}
where $(\mu_+, \mu_-)$ are two positive measures, admits an unique solution which is given by :
$$\rho_+(dz) =\Pb_{(x,0)}(X_{\zeta_{ab}}-b\in dz) = \frac{\sin(\pi \beta) }{\pi} \frac{(x-a)^\alpha (b-x)^\beta}{b-x+z}  \frac{ z^{-\beta}}{(b-a +z)^{\alpha}} dz$$
and
$$\rho_-(dz) =\Pb_{(x,0)}(a-X_{\zeta_{ab}}\in dz)=  \frac{\sin(\pi \alpha) }{\pi} \frac{(x-a)^\alpha (b-x)^\beta}{x-a+z}  \frac{ z^{-\alpha}}{(b-a +z)^{\beta}} dz.$$
%
%
%
\end{lemma}

\begin{proof}
We first check that the pair $(\rho_+, \rho_-)$ is solution of (\ref{eq:sys}). Recall the integral representation of the hypergeometric function \cite[p.317]{GrRy}:
\begin{equation}\label{eq:Gauss}
\int_0^{+\infty} \frac{z^{r-1}}{(u+z)^{p} (v+z)^{q} }dz = u^{-p}\, v^{r-q} B(r, p+q-r)  \pFq{2}{1}{p\quad r}{p+q}{1-\frac{v}{u}}
\end{equation}
for $u,v,r>0$ and  $p+q>r$.
On the one hand, using (\ref{eq:Gauss}) and setting $ \xi = \frac{b-x}{a-x}<0$, we have~:
\begin{multline}\label{eq:F1}
  \frac{\sin(\pi(\beta-s))}{\sin(\pi\beta)(b-x)^s} \int_0^{+\infty} z^s \rho_+(dz)  \\=
 \frac{\sin(\pi(\beta-s))}{\pi} (-\xi)^{-\alpha} \left(\frac{b-a}{b-x}\right)^{1+s-\nu} \frac{\Gamma(s+1-\beta)\Gamma(\nu-s)}{\Gamma(1+\alpha)} \pFq{2}{1}{1\quad s-\beta+1}{1+\alpha}{\frac{1}{\xi}}.
  \end{multline} 
On the other hand, still from  (\ref{eq:Gauss}), 
\begin{multline}\label{eq:F2}
  \frac{1}{(b-x)^s} \int_0^{+\infty} (b-a+z)^s \rho_-(dz) \\
 =   \frac{\sin(\pi\alpha)}{\pi}  (-\xi)^{1-\alpha} \left(\frac{b-a}{b-x}\right)^{1+s-\nu}  \frac{\Gamma(1-\alpha)\Gamma(\nu-s)}{\Gamma(1+\beta-s)}  \pFq{2}{1}{1\quad 1-\alpha}{1+\beta-s}{\xi}.
 \end{multline}
Applying the Gauss transformation (see \cite[Section 9.132]{GrRy}) on (\ref{eq:F2}), we deduce that
\begin{align}
\notag&(-\xi) \frac{\Gamma(1-\alpha)}{\Gamma(1+\beta-s)}  \pFq{2}{1}{1\quad 1-\alpha}{1+\beta-s}{\xi} \\
\notag &\qquad= \frac{\Gamma(-\alpha)}{ \Gamma(\beta-s)}\pFq{2}{1}{1\quad s-\beta+1}{1+\alpha}{\frac{1}{\xi}} + \frac{\Gamma(1-\alpha)\Gamma(\alpha)}{\Gamma(\nu-s)} (-\xi)^{\alpha} \pFq{2}{1}{1-\alpha \quad 1+s-\nu}{1-\alpha}{\frac{1}{\xi}}\\
\label{eq:F3} &\qquad= \frac{\Gamma(-\alpha)}{ \Gamma(\beta-s)}\pFq{2}{1}{1\quad s-\beta+1}{1+\alpha}{\frac{1}{\xi}} + \frac{\pi}{\sin(\pi \alpha)} \frac{ (-\xi)^{\alpha} }{ \Gamma(\nu-s)}\left(\frac{b-a}{b-x}\right)^{\nu-s-1}.
 \end{align}
 The result now follows by adding (\ref{eq:F1}) and (\ref{eq:F2}), using (\ref{eq:F3}) and the complement formula for the Gamma function. This proves that the first equation of (\ref{eq:sys}) is satisfied and similar computations show that the second equation also holds.\\
 
 \noindent
  Now, assume that $(\mu_+, \mu_-)$ is another solution of (\ref{eq:sys}). Then, using (\ref{eq:Xzetab}) and inverting the first Mellin transform of each equation, we deduce that the measures $(\mu_+, \mu_-)$ admit densities. Keeping the same notation for the densities, we obtain by difference
\begin{equation*}
\left\{\begin{array}{l}
 \displaystyle \mu_{+}(z) - \rho_+(z) =  \frac{\sin(\pi\beta)}{\pi} \left(\frac{z}{b-x}\right)^{-\beta}  \int_0^{+\infty} \frac{(b-a+y)^\beta(\rho_-(y) - \mu_-(y)) }{(b-x)(b-a+y)+z }dy\\
\vspace{-.3cm}\\
  \displaystyle \mu_{-}(z) - \rho_-(z) = \frac{\sin(\pi\alpha)}{\pi} \left(\frac{z}{x-a}\right)^{-\alpha}     \int_0^{+\infty} \frac{(b-a+y)^\alpha (\rho_+(y)-\mu_+(y))}{(x-a)(b-a+y)+z} dy\\
\end{array}\right.,
\quad z>0.
\end{equation*}
Let $\varepsilon \in (0, \min(\alpha, \beta))$. Integrating against $z^{-\varepsilon}$ on $(0,+\infty)$, we obtain the inequalities :
\begin{equation*}
\left\{\begin{array}{l}
 \displaystyle \int_0^{+\infty} z^{-\varepsilon} |\mu_{+}(z) - \rho_+(z)| dz   \leq \frac{\sin(\pi \beta)}{\sin(\pi(\beta+\varepsilon))}  \int_0^{+\infty}y^{-\varepsilon}|\rho_-(y) - \mu_-(y)| dy \\
\vspace{-.3cm}\\
 \displaystyle \int_0^{+\infty} z^{-\varepsilon} |\mu_-(z)-\rho_{-}(z)| dz   \leq \frac{\sin(\pi \alpha)}{\sin(\pi(\alpha+\varepsilon))}  \int_0^{+\infty}y^{-\varepsilon}|\rho_+(y) - \mu_+(y)| dy \\\end{array}\right.,
\end{equation*}
hence
$$ \int_0^{+\infty} z^{-\varepsilon} |\mu_{+}(z) - \rho_+(z)| dz \leq  \frac{\sin(\pi \beta)}{\sin(\pi(\beta+\varepsilon))}\frac{\sin(\pi \alpha)}{\sin(\pi(\alpha+\varepsilon))}  \int_0^{+\infty}y^{-\varepsilon}|\rho_+(y) - \mu_+(y)| dy.$$
Recall finally that $\alpha$ and $\beta$ are positive and smaller than $\nu < \frac{1}{2}$. Therefore, we may take $\varepsilon$ small enough so that  $\displaystyle  \frac{\sin(\pi \beta)}{\sin(\pi(\beta+\varepsilon))}\frac{\sin(\pi \alpha)}{\sin(\pi(\alpha+\varepsilon))}<1$ and this implies that $\mu_+ = \rho_+$ a.s., and thus also $\mu_- = \rho_-$ a.s.
\end{proof}

\subsection{Proof of Theorem \ref{theo:B} when $y=0$}
To retrieve the distribution of $Y_{T_{ab}}$ observe that applying the Markov property and using Lemma \ref{lem:Xs0}, we have 
\begin{align*}
\rho_+(dz) = \Pb_{(x,0)}\left(X_{\zeta_{ab}}-b\in dz\right) &=  \E_{(x,0)}\left[\Pb_{(0, Y_{T_{ab}})}\left(X_{\sigma_0}\in dz\right) \Un_{\{Y_{T_{ab}}>0\}} \right]\\
&=
\frac{z^{-\nu-1}}{\Gamma\left(\nu\right)A^\nu}  \E_{(x,0)}\left[Y_{T_{ab}}^{2-\delta} \exp\left(- \frac{Y_{T_{ab}}^{2+\gamma} }{A z} \right) \Un_{\{Y_{T_{ab}}>0\}}\right] dz
\end{align*}
hence, the modified Laplace transform of $Y_{T_{ab}}$ is given by :
$$\E_{(x,0)}\left[Y_{T_{ab}}^{2-\delta} e^{- \frac{\lambda }{A} Y_{T_{ab}}^{2+\gamma}}\Un_{\{Y_{T_{ab}}>0\}}\right] =\frac{\sin(\pi \beta)}{\pi} \frac{A^\nu \Gamma(\nu)}{(\lambda(b-a) +1)^{\alpha}}\frac{(x-a)^\alpha (b-x)^\beta}{\lambda(b-x)+1}
 $$
 and similarly 
$$
\E_{(x,0)}\left[|Y_{T_{ab}}|^{2-\delta} e^{- \frac{ \lambda }{A}  |Y_{T_{ab}}|^{2+\gamma}}\Un_{\{Y_{T_{ab}}<0\}}\right] =c^{1-\alpha}  \frac{\sin(\pi \alpha)}{\pi} \frac{A^\nu\Gamma(\nu)}{\left(\lambda (b-a) +c\right)^{\beta}}\frac{(x-a)^\alpha (b-x)^\beta}{\lambda(x-a)+c}.
$$
In the Brownian case  (when $\delta=1$), these modified Laplace transforms are the exact same ones obtained by Lachal \cite{LacExit2}. 
Theorem \ref{theo:B} now follows by inverting these two transforms, using for instance \cite[p.238, Formula (8)]{Erd}.\\

\subsection{Proof of Theorem \ref{theo:B} for $y\neq0$}\label{sec:3.3}

We give the proof of the probability distribution of $Y_{T_{ab}}$ only for $z>0$, the case $z<0$ being similar. Assume first that $y>0$. We decompose
\begin{align*}
\notag \E_{(x,y)}\left[Y_{T_{ab}}^{s(2+\gamma)}\Un_{\{Y_{T_{ab}}>0\}}\right]&= \E_{(x,y)}\left[Y_{T_{ab}}^{s(2+\gamma)}\Un_{\{Y_{T_{ab}}>0, \, \sigma_0< T_{ab}\}}\right] + \E_{(x,y)}\left[Y_{T_{ab}}^{s(2+\gamma)}\Un_{\{Y_{T_{ab}}>0,\,  T_{ab}\leq \sigma_0\}}\right]\\
&=\E_{(x,y)}\left[ \E_{(X_{\sigma_0},0)}\left[Y_{T_{ab}}^{s(2+\gamma)}\Un_{\{Y_{T_{ab}}>0\}}\right] \Un_{\{X_{\sigma_0}<b\}}\right] + \E_{(x,y)}\left[Y_{T_{b}}^{s(2+\gamma)}\Un_{\{T_b\leq \sigma_0\}}\right].
\end{align*}
The first Mellin transform of the right hand-side may easily be inverted using Theorem \ref{theo:B} and the distribution of $X_{\sigma_0}$ given in Lemma \ref{lem:Xs0}. 
For the second one, notice that 
\begin{align*}
 \E_{(x,y)}\left[ Y_{T_b}^{s(2+\gamma)}\Un_{\{ T_b\leq \sigma_0\}}\right] &= \E_{(x,y)}\left[ Y_{T_b}^{s(2+\gamma)}\right] -  \E_{(x,y)}\left[ Y_{T_b}^{s(2+\gamma)}\Un_{\{\sigma_0 < T_b\}}\right]\\
 &=\E_{(x,y)}\left[ Y_{T_b}^{s(2+\gamma)}\right] - \E_{(x,y)}\left[   \E_{(X_{\sigma_0},0)}\left[Y_{T_b}^{s(2+\gamma)}\right] \Un_{\{X_{\sigma_0}< b\}}\right].
\end{align*}
Therefore, using (\ref{eq:MelYT}) and (\ref{eq:Xzeta}), this implies that 
\begin{equation*}\label{eq:MelYT0}
 \E_{(x,y)}\left[ Y_{T_b}^{s(2+\gamma)}\Un_{\{T_b\leq \sigma_0\}}\right] =A^{s} \frac{\Gamma(\nu)}{\Gamma\left(\nu-s  \right)}  \E_{(x,y)}\left[(X_{\sigma_0}-b)_+^{s}\right],
  \end{equation*}
and we deduce from Lemma \ref{lem:Xs0} that 
\begin{align*}
 \E_{(x,y)}[Y_{T_{b}}^{s(2+\gamma)}\Un_{\{T_b\leq \sigma_0\}}] & =\frac{A^{s-\nu} y^{2-\delta}}{\Gamma(\nu-s)} \int_{b-x}^{+\infty} (z+x-b)^s z^{-\nu-1} e^{-\frac{y^{2+\gamma}}{Az}}dz\\
\notag &= y^{2-\delta} \frac{\Gamma(s+1)}{\Gamma(\nu+1)} (A(b-x))^{s-\nu} \pFq{1}{1}{\nu-s}{\nu+1}{- \frac{y^{2+\gamma}}{A(b-x)}} \\
 \notag &=y^{2-\delta} \frac{\Gamma(s+1)}{\Gamma(\nu+1)} (A(b-x))^{s-\nu} e^{- \frac{y^{2+\gamma}}{A(b-x)}}\pFq{1}{1}{1+s}{\nu+1}{ \frac{y^{2+\gamma}}{A(b-x)}}.
\end{align*}
This expression may finally be inverted thanks to \cite[p.364 Formula (24)]{Erd}, with $\xi>0$ :
$$\Gamma(s+1) \pFq{1}{1}{1+s}{\nu+1}{\xi} = \Gamma(\nu+1) \int_0^{+\infty} z^{s}\, (\xi z)^{-\frac{\nu}{2}} e^{-z} I_\nu(2\sqrt{\xi z}) dz.$$
Next, when $y<0$, we must have $\sigma_0 \leq T_b$. Therefore, 
\begin{align*}
\notag \E_{(x,y)}\left[Y_{T_{ab}}^{s(2+\gamma)}\Un_{\{Y_{T_{ab}}>0\}}\right]&= \E_{(x,y)}\left[Y_{T_{ab}}^{s(2+\gamma)}\Un_{\{Y_{T_{ab}}>0, \, \sigma_0< T_{ab}\}}\right] \\
&=\E_{(x,y)}\left[ \E_{(X_{\sigma_0},0)}\left[Y_{T_{ab}}^{s(2+\gamma)}\Un_{\{Y_{T_{ab}}>0\}}\right] \Un_{\{X_{\sigma_0}>a\}}\right]
\end{align*}
and we conclude again by using Theorem \ref{theo:B} and Lemma \ref{lem:Xs0}.
\noindent
 \qed

\subsection{Proof of Corollary \ref{cor:D}}

Assume first that $y=0$. Then, from Lemma \ref{lem:lawXzetaab} and (\ref{eq:Gauss}), we have 
\begin{align*}
\Pb_{(x,0)}\left(T_b< T_a\right) &= \Pb_{(x,0)}\left(X_{\zeta_{ab}}-b\geq0 \right)\\
&=\frac{\sin(\pi \beta) }{\pi} (x-a)^\alpha (b-x)^\beta\int_0^{+\infty}    \frac{ z^{-\beta}(b-a +z)^{-\alpha}}{b-x+z} dz\\
&= \frac{\sin(\pi \beta) }{\pi} \frac{\Gamma(1-\beta)\Gamma\left(\nu\right)}{\Gamma(1+\alpha)} \left(\frac{x-a}{b-a}\right)^{\alpha} 
 \pFq{2}{1}{\alpha\quad 1-\beta}{1+\alpha}{\frac{x-a}{b-a}}
\end{align*}
and the expression given in Corollary  \ref{cor:D} is a consequence  of the compensation formula for the Gamma function. When $y\neq 0$, the result follows by noticing that 
$$\Pb_{(x,y)}\left(T_b< T_a\right) = \Pb_{(x,y)}(Y_{T_{ab}}>0)$$
and taking $s=0$ in the previous Mellin transforms of Section \ref{sec:3.3}.
\qed

\bigskip

\noindent
{\bf Acknowledgements.} We wish to thank T. Simon for his comments on an earlier version of this manuscript.

%


\begin{thebibliography}{10}


\bibitem{AlAy}
L. Alili and A. Aylwin. On the semi-group of a scaled skew Bessel process. {\em Statist. Probab. Lett.} {\bf 145}, 96--102, 2019.

\bibitem{AuSi} 
F.~Aurzada and T. Simon. Persistence probabilities and exponents.  {\em L\'evy Matters V}, Lecture Notes in Math.  {\bf 2149}, Springer, 183-224,   2015.

\bibitem{Ble}
S. Blei. On symmetric and skew Bessel processes. {\em Stochastic Process. Appl.} {\bf 122} (9), 3262-3287, 2012.

\bibitem{BMS}
A.~J.~Bray, S.~N.~Majumdar and G.~Schehr. Persistence and first-passage properties in non-equilibrium systems. {\em Adv. Physics} {\bf 62} (3), 225-361, 2013.

\bibitem{Cet}
U. \c{C}etin. On certain integral functionals of squared Bessel processes. {\em Stochastics} {\bf 87} (6) 1033-1060, 2015.

\bibitem{Erd}
A.~Erd\'elyi, W.~Magnus, F.~Oberhettinger and F. G.~Tricomi.  Tables of integral transforms. Vol. II.  {\em McGraw-Hill Book Company, Inc., New York-Toronto-London}, 1954.

\bibitem{FGD}
P. Flajolet, X. Gourdon and  P. Dumas. Mellin transforms and asymptotics: harmonic sums. {\em Theoret. Comput. Sci.} {\bf 144} (1-2), 3-58, 1995.

\bibitem{GoYo}
A. G\"oing-Jaeschke and M. Yor. A survey and some generalizations of Bessel processes. {\em Bernoulli}  {\bf 9}, 313-349,  2003.

\bibitem{GrRy}
I.S. Gradshteyn and I.M. Ryzhik. Table of integrals, series, and products. {\em Elsevier/Academic Press, Amsterdam}, 2007.

\bibitem{GJW}
P.~Groeneboom, G.~Jongbloed, and J.~A. Wellner. Integrated {B}rownian motion, conditioned to be positive. {\em Ann. Probab.}, {\bf 27} (3), 1283-1303, 1999.

\bibitem{IsKo}Y.~Isozaki and S.~Kotani. Asymptotic estimates for the first hitting time of fluctuating additive functionals of Brownian motion. {\em S\'emininaire de ProbababilitŽs} XXXIV,  Lecture Notes in Math. {\bf 1729}, 374-387, 2000.
 
 
 \bibitem{JanArea}
S.~Janson. Brownian excursion area, Wright's constants in graph enumeration, and other Brownian areas. {\em Probab. Surv.} {\bf 4},  80-145, 2007.  
 

\bibitem{LacT0}
A.~Lachal. Sur le premier instant de passage de l'int\'egrale du mouvement brownien. {\em Ann. Inst. H. Poincar\'e, Probab. Stat.} {\bf 27} (3), 385-405, 1991.

\bibitem{LacTn}
A. Lachal. Les temps de passage successifs de l'int\'egrale du mouvement brownien. {\em Ann. Inst. H. Poincar\'e Probab. Stat.} {\bf  33} (1), 1-36, 1997.

\bibitem{LacExit1}
A.~Lachal.  First exit time from a bounded interval for a certain class of additive functionals of Brownian motion. {\em J. Theoret. Probab.} {\bf 13} (3), 733-775, 2000. 

\bibitem{LacExit2}
A.~Lachal. Some explicit distributions related to the first exit time from a bounded interval for certain functionals of Brownian motion. {\em J. Theoret. Probab.}{\bf 19} (4), 757-771, 2006.

\bibitem{Lej}
A. Lejay.  On the constructions of the skew Brownian motion. {\em Probab. Surv.} {\bf 3}, 413-466,  2006.

\bibitem{Pro}
C.~Profeta. Some limiting laws associated with the integrated Brownian motion. {\em ESAIM Probab. Statist.} {\bf  19}, 148-171, 2015.

\bibitem{PSPers}
C. Profeta and T. Simon. Persistence of integrated stable processes. {\em Probab. Theory Relat. Fields.} {\bf 162} (3),  463-485, 2015.

\bibitem{PSAlea}
C. Profeta and T. Simon. Windings of the stable Kolmogorov process. {\em  ALEA} {\bf  XII},  115-127, 2015.

\bibitem{Yor}
M. Yor. Exponential functionals of Brownian motion and related processes. {\em Springer Finance, Springer-Verlag}, Berlin, 2001.
\end{thebibliography}
\end{document}